\documentclass[10pt]{amsart}
\usepackage{amsmath}
\usepackage{amssymb}
\usepackage{mathtools}      
\usepackage{mathabx} 
\usepackage{pbsi} 
\usepackage{enumitem}
\usepackage{amsthm}
\usepackage{thmtools}
\usepackage{tikz,pgfplots}
\pgfplotsset{compat=1.18}
\usetikzlibrary{decorations.text}
\usetikzlibrary{arrows.meta}
\usetikzlibrary{decorations.pathmorphing}
\usepackage[font=small]{caption}
\usepackage[colorlinks=true,backref=page,bookmarksopen=true]{hyperref} 
\usepackage[capitalise]{cleveref} 

\usepackage{xcolor}
\hypersetup{
    colorlinks,
    linkcolor={red!50!black},
    citecolor={blue!50!black},
    urlcolor={blue!80!black}
}

\renewcommand*{\backref}[1]{}

\renewcommand*{\backrefalt}[4]{\ \tiny 
  \ifcase #1 ({\color{red} \bf NOT CITED.})%
  \or    ($\uparrow$#2)%
  \else   ($\uparrow$#2)%
  \fi}
  
\declaretheorem{theorem} 
\declaretheorem[sibling=theorem]{proposition} 

\declaretheorem[sibling=theorem]{lemma}

\declaretheorem[sibling=theorem, style=remark]{remark}

\hypersetup{bookmarksdepth = 3} 

\setlist[enumerate,1]{label={\upshape(\alph*)},ref=\alph*}

\newcommand{\N}{\mathbb{N}} \newcommand{\Z}{\mathbb{Z}} \newcommand{\Q}{\mathbb{Q}} \newcommand{\R}{\mathbb{R}} 
\newcommand{\T}{\mathbb{T}}

\newcommand{\st}{\;\mathord{:}\;}

\renewcommand{\setminus}{\smallsetminus}

\renewcommand{\epsilon}{\varepsilon}

\newcommand{\arxiv}[2]{\href{http://arxiv.org/abs/#1}{arXiv: {#1} [{#2}]}}
\newcommand{\doi}[1]{\href{http://doi.org/#1}{\tt doi}} 
\newcommand{\directlink}[1]{\href{#1}{\tt URL}}
\newcommand{\MRev}[1]{\href{https://mathscinet.ams.org/mathscinet-getitem?mr=#1}{\tt MR}} 
\newcommand{\Zbl}[1]{\href{https://zbmath.org/?q=an:#1}{\tt Zbl}} 

\DeclareFontFamily{U} {MnSymbolA}{} 
\DeclareFontShape{U}{MnSymbolA}{m}{n}{
   <-6> MnSymbolA5
   <6-7> MnSymbolA6
   <7-8> MnSymbolA7
   <8-9> MnSymbolA8
   <9-10> MnSymbolA9
   <10-12> MnSymbolA10
   <12-> MnSymbolA12}{}
\DeclareFontShape{U}{MnSymbolA}{b}{n}{
   <-6> MnSymbolA-Bold5
   <6-7> MnSymbolA-Bold6
   <7-8> MnSymbolA-Bold7
   <8-9> MnSymbolA-Bold8
   <9-10> MnSymbolA-Bold9
   <10-12> MnSymbolA-Bold10
   <12-> MnSymbolA-Bold12}{}
\DeclareSymbolFont{MnSyA} {U} {MnSymbolA}{m}{n}
\DeclareFontFamily{U} {MnSymbolC}{}
\DeclareFontShape{U}{MnSymbolC}{m}{n}{
  <-6> MnSymbolC5
  <6-7> MnSymbolC6
  <7-8> MnSymbolC7
  <8-9> MnSymbolC8
  <9-10> MnSymbolC9
  <10-12> MnSymbolC10
  <12-> MnSymbolC12}{}
\DeclareFontShape{U}{MnSymbolC}{b}{n}{
  <-6> MnSymbolC-Bold5
  <6-7> MnSymbolC-Bold6
  <7-8> MnSymbolC-Bold7
  <8-9> MnSymbolC-Bold8
  <9-10> MnSymbolC-Bold9
  <10-12> MnSymbolC-Bold10
  <12-> MnSymbolC-Bold12}{}
\DeclareSymbolFont{MnSyC} {U} {MnSymbolC}{m}{n}

\DeclareMathSymbol{\top}{\mathord}{MnSyA}{219} 
\DeclareMathSymbol{\bot}{\mathord}{MnSyA}{217}
\DeclareMathSymbol{\smallplus}{\mathord}{MnSyC}{20} 
\DeclareMathSymbol{\smallminus}{\mathord}{MnSyC}{16} 
\DeclareMathSymbol{\smalltimes}{\mathord}{MnSyC}{21}
\DeclareMathSymbol{\smallpm}{\mathord}{MnSyC}{22} 
\DeclareMathSymbol{\smallmp}{\mathord}{MnSyC}{23}

\usepackage{calc} 

\definecolor{MapleRed}{RGB}{120,0,14}

\begin{document}

\title{An isolated Lyapunov exponent}
\author{Jairo Bochi}
\address{Department of Mathematics, The Pennsylvania State University}
\email{\href{mailto:bochi@psu.edu}{bochi@psu.edu}}
\date{September 16, 2025}

\begin{abstract}
We construct a continuous linear cocycle over an expanding base dynamics for which the Lyapunov exponents of all ergodic invariant probability measures are small, except for one measure whose Lyapunov exponents are away from zero. The support of this distinguished measure is not a periodic orbit and therefore our example violates the periodic approximation property. 
\end{abstract}

\maketitle

\section{Introduction}

\subsection{Lyapunov exponents}

A \emph{linear cocycle} is a pair $(T,A)$ where $T \colon X \to X$ is a map from a set $X$ to itself, and $A \colon X \to \mathrm{Mat}(d,\R)$ is a function taking values in the set of real $d \times d$ matrices. The \emph{cocycle products} are formed by multiplying the values of $A$ along orbits of $T$, that is,
\begin{equation}
A_T^{(n)}(x) \coloneqq A(T^{n-1}x) \cdots A(x), \quad \text{for all $x \in X$ and $n \ge 0$,}
\end{equation}
where the empty product $A^{(0)}(x)$ is defined as the identity matrix.

Now assume that $T$ is a measure-preserving transformation of a probability space $(X,\mathcal{S},\mu)$, 
that the function $A$ takes values on the group of invertible matrices $\mathrm{GL}(d,\R)$, 
and that the (nonnegative) function $\log \max\{\|A(x)\|,\|A(x)^{-1}\|\}$ is $\mu$-integrable. Then the \emph{Lyapunov exponents}
\begin{equation}
\lambda_i(T,A,x) \coloneqq \lim_{n \to \infty} \frac{1}{n} \log s_i \big( A_T^{(n)}(x) \big)
\end{equation}
exist for $\mu$-almost every $x \in X$ and every $i \in \{1,\dots,d\}$; here $s_i (\mathord{\cdot})$ denotes the $i^\text{th}$ largest singular value.  
The average values $\lambda_i(T,A,\mu) \coloneqq \int_X \lambda_i(T,A, x) \, d\mu(x)$ are also well defined. 
If the measure $\mu$ is ergodic, then $\lambda_i(T,A,x) = \lambda_i(T,A,\mu)$ for $\mu$-almost every $x$.

A case of particular interest is when the measure $\mu$ is supported on a periodic orbit, that is $\mu = \frac{1}{k} \sum_{j=0}^{k-1} \delta_{T^j p}$ where $p \in X$ and $k>0$ are such that $T^k p = p$. In this case, the Lyapunov exponents are $\lambda_i(T,A,\mu) = \frac{1}{k} \log |\beta_i|$, where $\beta_1 , \dots, \beta_d$ are the eigenvalues of the matrix product $A_T^k(p)$, repeated according to multiplicity and ordered according to their absolute values. 

\medskip

Henceforth,  $X$ will be a compact metric space, and the transformation~$T$ will be continuous. The set of all $T$-invariant Borel probability measures on $X$ will be denoted $\mathcal{M}_T$ and endowed with the weak topology. This is a convex compact set, and the set of its extreme points will be denoted by $\mathcal{M}_T^\mathrm{erg}$, since it is composed exactly by the ergodic measures. Also, let $\mathcal{M}_T^\mathrm{per}$ denote the subset of $\mathcal{M}_T^\mathrm{erg}$ formed by measures whose support is a periodic orbit.

Now, if $A \colon X \to \mathrm{GL}(d,\R)$ is a continuous function on the compact space $X$, then the Lyapunov exponents $\lambda_i(T,A,\mu)$ are well-defined for all $\mu \in \mathcal{M}_T$. Therefore, fixed a continuous linear cocycle $(T,A)$ over a compact base, we can consider the averaged Lyapunov exponents $\lambda_i(T,A,\mathord{\cdot})$ as functions on the space $\mathcal{M}_T$ of invariant measures. 
For $i=1$, Kingman's subadditive ergodic theorem implies that, for all $\mu \in \mathcal{M}_T$,
\begin{equation}
\lambda_1(T,A,\mu) = \inf_{n>0} \frac{1}{n} \int_X \log \|  A_T^{(n)}(x) \| \, d\mu(x) \, ;
\end{equation}
in particular, the function $\lambda_1(T,A,\mathord{\cdot})$ is upper semicontinuous.
A similar argument using exterior powers shows that $\sum_{j=1}^i\lambda_j(T,A,\mathord{\cdot})$ is upper semicontinuous, for each~$i$.

In general, the Lyapunov exponents are \emph{not} continuous functions of the measure. For example, let $T = \sigma$ be the one-sided full shift on two symbols $0$, $1$, and for each $x = (x_i)_{i \ge 0}$, let the matrix $A(x)$ be $\left( \begin{smallmatrix} 2 & 0 \\ 0 & 1/2 \end{smallmatrix} \right)$ if the zeroth symbol $x_0$ is $0$ and $\left( \begin{smallmatrix} 0 & -1 \\ 1 & 0 \end{smallmatrix} \right)$ otherwise. For each $k \ge 1$, let $\mu_k$ be the shift-invariant measure supported on the orbit of the periodic point $(0^{k-1} 1)^\infty$ (that is, the point $p = (p_i)_{i \ge 0}$ such that $p_i = 1$ iff $i \equiv 1 \bmod{k}$). Then the sequence $(\mu_k)$ converges weakly to $\nu \coloneqq \delta_{0^\infty}$. However, $\lambda_1(\sigma,A,\mu_k) = 0$ for every $k$ and $\lambda_1(\sigma,A,\nu) = \log 2$. So $\nu$ is a point of discontinuity of $\lambda_1(\sigma,A,\mathord{\cdot})$.

For locally constant cocycles over a shift map (like the example above), in arbitrary dimension, the Lyapunov exponents are continuous on the subset of fully-supported \emph{Bernoulli} measures: see \cite{AEV}. Under extra conditions, the Lyapunov exponents become analytic functions: see \cite{Peres,ADM}. We refer the reader to \cite{Viana20} for a detailed survey of continuity (or lack thereof) of Lyapunov exponents in various contexts, and to \cite{DDGK} for recent developments. 

\subsection{An isolated Lyapunov exponent}

The purpose of this paper is to construct examples of continuous linear cocycles $(T,A)$ for which the Lyapunov exponents display a rather extreme form of discontinuity as functions of the measure.
Our examples will take values in the group $\mathrm{SL}(2,\R)$.
In particular, it will be sufficient to consider the first Lyapunov exponent $\lambda_1$, since the second one will be $\lambda_2 = - \lambda_1$. 
Our main result is: 

\begin{theorem}\label{t.main}
Let $D$ be the doubling map on the circle $\T \coloneqq \R / \Z$. 
Fix numbers $c>\epsilon>0$.
There exists a continuous map $A \colon \T \to \mathrm{SL}(2,\R)$ with the following properties:
\begin{enumerate}
\item\label{i.main1}
there exists an ergodic measure $\nu$ such that $\lambda_1(D,A,\nu) = c$;

\item\label{i.main2} 
for all ergodic measures $\mu$ different from $\nu$, we have $\lambda_1(D,A,\mu) \le \epsilon$;

\item\label{i.main3}  
if $(\mu_k)$ is any sequence of ergodic measures different from $\nu$ and converging weakly to $\nu$, we have 
$\lim\limits_{k \to \infty} \lambda_1(D,A,\mu_k) = 0$.
\end{enumerate}
In fact, $\nu$ can be chosen as a Sturmian measure with arbitrary irrational parameter.
\end{theorem}

This is illustrated in \cref{f.jump}. 
The definition of Sturmian measures is given in \cref{s.sturmian}; for the moment, we note that they are not supported on periodic orbits (at least when the parameter is irrational).

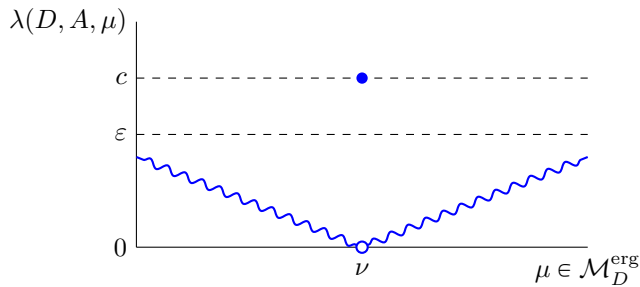
\begin{figure}
\begin{tikzpicture}[scale=3]
\draw (-1,0) -- (1,0) node[below]{$\mu \in \mathcal{M}^\mathrm{erg}_D$};
\draw (-1,0) node[left]{$0$} -- (-1,1) node[left]{$\lambda(D,A,\mu)$};  
\draw[dashed] (-1,.5) node[left]{$\epsilon$} -- (1,.5);
\draw[dashed] (-1,.75) node[left]{$c$} -- (1,.75);
\tikzset{decoration={snake, amplitude=0.5mm, segment length=2.5mm, post length=0mm, pre length=1mm}}
\draw[MapleRed, thick, decorate] (0,0) -- (1,.4);
\draw[MapleRed, thick, decorate] (0,0) -- (-1,.4);
\fill[thick, draw=MapleRed, fill=white] (0,0) circle (0.025);
\fill[thick, MapleRed] (0,.75) circle (0.025);
\draw (0,-.03) node[below] {$\nu$};
\end{tikzpicture}
\caption{The Lyapunov exponent function in \cref{t.main}.}\label{f.jump}
\end{figure}

In the next subsection, we describe an application of \cref{t.main}.

\subsection{Periodic approximation of Lyapunov exponents}

Subshifts of finite type have the property that 
every ergodic Borel probability measure can be approximated by a probability measure supported on a periodic orbit,
that is, $\mathcal{M}_T^\mathrm{per}$ is dense in $\mathcal{M}_T^\mathrm{erg}$.  
This periodic approximation property also holds for Axiom A diffeomorphisms \cite{Sigmund} and $C^1$-generic diffeomorphisms \cite[Theorem~4.2]{ABC}. 
Katok's Closing Lemma \cite[Section~3]{Katok} shows that every ergodic \emph{hyperbolic} measure of a $C^{1+\alpha}$ diffeomorphism can be approximated by measures supported on periodic orbits (see \cite{LLiuS} for further information).

Given a linear cocycle $(T,A)$ and an ergodic measure $\mu \in \mathcal{M}_T^\mathrm{erg}$, we are interested in the problem of approximating  $\mu$ by elements of $\mathcal{M}_T^\mathrm{per}$ and simultaneously approximating the Lyapunov exponents of the cocycle. 
The following result, due to Kalinin \cite{Kalinin}, provides sufficient conditions for this problem to have a solution:

\begin{theorem}[Kalinin's periodic approximation theorem] \label{t.PAT}
Let $T \colon X \to X$ be either a subshift of finite type or the restriction of a diffeomorphism to a hyperbolic set with local product structure. 
Suppose that $A \colon X \to \mathrm{GL}(d,\R)$ is H\"older-continuous.
Then, for every $\mu \in \mathcal{M}_T^\mathrm{erg}$, there exists a sequence $(\mu_k)$ in $\mathcal{M}_T^\mathrm{per}$ such that $\mu_k \to \mu$ weakly and $\lambda_i(T,A,\mu_k) \to \lambda_i(T,A,\mu)$ as $k \to \infty$.
\end{theorem}

We observe that the original statement in \cite[Theorem~1.4]{Kalinin} presents a slightly more general hypothesis on the dynamics. Moreover, it does not explicitly include the conclusion that $\mu_k \to \mu$, although this property follows from the proof. \cref{t.PAT}  was preceded by a more restrictive result of Wang and Sun \cite[Theorem~3.1]{WangSun} and was independently established by Dai \cite{Dai_Forum}. 

The Periodic Approximation Theorem~\ref{t.PAT} has several applications. 
It is a key ingredient in Kalinin's proof of the Livsic theorem for matrix cocycles \cite[Theorem~1.1]{Kalinin}. 
Using the connection between the joint spectral radius and Lyapunov exponents (see \cite[Theorem~3.1]{DHX}, \cite[Theorem~2.1]{Morris13}), the well-known theorem of Berger and Wang \cite{BergerWang} becomes an immediate corollary of \cref{t.PAT}.

Improvements of \cref{t.PAT} in various directions can be found in the articles \cite{Backes17,Backes18,Dai_Nonlin,KalSad_DCDS,LLiaoS,ZC}. Related results of perturbative nature include \cite[Theorem~3.18]{ABC} and \cite[Theorem~5.1]{BoBo}. The problem of extending the Periodic Approximation Theorem~\ref{t.PAT} to infinite-dimensional cocycles is subtle, and there are positive and negative results: see \cite{BackesDrag, Hurtado, KalSad_ETDS}. For a survey on periodic approximation of Lyapunov exponents, see \cite[Section~13.5]{Sad24}.

In spite of this activity, the necessity of the hypothesis of H\"older continuity in \cref{t.PAT} has not been investigated before.
By lifting the cocycle constructed in \cref{t.main} to the natural extension of the doubling map (i.e., the Smale--Williams solenoid), 
we conclude that the Periodic Approximation Theorem~\ref{t.PAT} does not hold for continuous cocycles. 
This argument also shows that the cocycle $A$ in \cref{t.main} cannot be H\"older-continuous.

For another instance of exotic behavior of the Lyapunov exponents of continuous cocycles and further discussion of the role of regularity, see \cite{Bochi_mono}.

\subsection{A few words about the construction}

The example in \cref{t.main} is constructed as follows. Let $\nu$ be a Sturmian measure with irrational parameter, and let $K$ be its support. Then the restriction of the doubling map to $K$ is semiconjugate to an irrational rotation. Using this fact, we start by defining the cocycle $A$ on $K$ so that it is nonuniformly hyperbolic there. Then we need to extend the cocycle to the complement of $K$ so that the Lyapunov exponents are small. The dynamics on the complement of $K$ has a convenient structure which we exploit. The key ingredient in our proof is a theorem from \cite{ABD12}, which among other things says that every nonuniformly hyperbolic $\mathrm{SL}(2,\R)$-cocycle over an irrational rotation can be ``accessed'' by a continuous path formed by cocycles with zero Lyapunov exponent. Our construction consists of ``inserting'' this path into the gaps of the Cantor set~$K$. However, to ensure continuity of the resulting cocycle, we need to ``modulate'' the small gaps. This has an effect on the Lyapunov exponents, which needs to be kept small.

\section{Preliminaries on linear cocycles}

Let $(T,A)$ be a continuous linear cocycle. 
Suppose $A$ takes values in the group $\mathrm{SL}^{\smallpm}(d,\R)$ of matrices with determinant $\pm 1$.
We say that the cocycle $(T,A)$ is \emph{product bounded} if  
\begin{equation}
\sup_{x \in X}  \sup_{n \ge 0} \| A_T^{(n)}(x) \| < \infty \, .
\end{equation}
This condition is satisfied if the cocycle $(T,A)$ is \emph{conjugate} to an orthogonal cocycle $(T,O)$, that is, 
\begin{equation}
A(x) = C(Tx) \, O(x) \, C(x)^{-1}
\end{equation} 
for some continuous maps $C \colon X \to \mathrm{GL}(d,\R)$ and $O \colon X \to \mathrm{O}(d)$.
The converse holds if the base dynamics $T$ is minimal: see \cite[Theorem~A]{CNP}.

We say that a continuous $\mathrm{SL}^{\smallpm}(2,\R)$-cocycle $(T,A)$ is \emph{uniformly hyperbolic} if 
\begin{equation}\label{e.UH}
	\liminf_{n \to \infty} \frac{1}{n} \inf_{x \in X} \log  \|  A_T^{(n)}(x) \|  > 0 \, .
\end{equation}
For further discussion and equivalent formulations, see \cite[{\S}3.8.1]{DamanikFv1} and \cite[{\S}2.2]{Viana14}.

\medskip

The Haar measure on the circle $\T = \R/\Z$ is denoted by $\mathrm{Leb}_{\T}$ or $\mathrm{Leb}$.
For each $\alpha \in \R$, let $R_\alpha \colon \T \to \T$ be the translation $x \mapsto x + \alpha \bmod{1}$. 

The starting point of our construction is the following:

\begin{proposition}\label{p.NUH}
Let $\alpha \in \R \setminus \Q$ and $c>0$.
There exists a continuous function $A \colon \mathbb{T} \to \mathrm{SL}(2,\R)$ such that the cocycle $(R_\alpha,A)$ is not uniformly hyperbolic and has Lyapunov exponent $\lambda_1(R_\alpha,A,\mathrm{Leb}) = c$.
\end{proposition}

The statement above is well-known and can be proved in several ways (see e.g.\ \cite[Chapter~9]{DamanikFv2}).
For the convenience of the reader, we exhibit a specific example of cocycle with the required properties, following Herman \cite{Herman81}.

\begin{proof}
Let $\gamma>1$ be a parameter.
Define $A \colon \T \to \mathrm{SL}(2,\R)$ by 
\begin{equation}\label{e.Herman_def} 
A(x) \coloneqq 
\begin{pmatrix} 
	\gamma & 0 \\ 
	0 & \gamma^{-1} 
\end{pmatrix}   
U(x)
\, , \quad \text{where }
U(x)\coloneqq 
\begin{pmatrix} 
	\cos 2\pi x & -\sin 2\pi x \\ 
	\sin 2\pi x & \phantom{-}\cos 2\pi x  
\end{pmatrix} \, .
\end{equation}
Herman \cite{Herman81} proved that the cocycle $(R_\alpha,A)$ has a positive Lyapunov exponent; actually, this exponent is given by the following formula (see \cite[Example~12]{AvilaBochi}):
\begin{equation}
	\lambda_1(R_\alpha,A,\mathrm{Leb}) = \log \frac{\gamma+\gamma^{-1}}{2} \, .
\end{equation}
In particular, the value of $\gamma>1$ can be adjusted so that $\lambda_1(R,A,\mathrm{Leb}) = c$.

Herman's cocycle $(R_\alpha, A)$ is not uniformly hyperbolic: this fact follows from a topological argument \cite[Proposition~4.2]{Herman83}. 
Let us present here a more direct proof.
We claim that, for all $n \ge 0$,
\begin{equation}
A_{R_\alpha}^{(2n)} (R_\alpha^{-n}(\tfrac{1}{2})) = (-1)^n U(-n\alpha) \, ;
\end{equation}
in particular, property \eqref{e.UH} fails.
The verification, by induction, is left to the reader. 
\end{proof}

A cocycle $(R_\alpha,A)$ satisfying the conclusions of \cref{p.NUH} is called \emph{nonuniformly hyperbolic}, because it has nonzero Lyapunov exponents and is not uniformly hyperbolic. 
According to \cite[Theorem~C]{Bochi02}, there exists a $C^0$-perturbation $\tilde{A}$ of $A$ such that $\lambda_1(R_\alpha,\tilde{A},\mathrm{Leb}) = 0$. 
(See \cite[{\S}9.2]{Viana14} for a variation of the proof).
On the other hand, \cite[Theorem~1]{ABD09} provides a perturbation $\tilde A$ which is conjugate to an orthogonal cocycle.
(A simplification of the proof was found in \cite{BNavas}).
The following result, which is \cite[Theorem~4]{ABD12} specialized to irrational rotations, provides even stronger conclusions: the perturbations above can be found in a continuous path converging to the original cocycle. 

\begin{theorem}\label{t.ABD}
Let $\alpha \in \R \setminus \Q$ and let $A \colon \mathbb{T} \to \mathrm{SL}(2,\R)$ be continuous. 
If the cocycle $(R_\alpha,A)$ is not uniformly hyperbolic, then there exist continuous maps 
\begin{align}
(x,t) \in \T \times [0,1] &\mapsto A_t(x) \in \mathrm{SL}(2,\R) \, , \\
(x,t) \in \T \times (0,1] &\mapsto C_t(x) \in \mathrm{SL}(2,\R) \, , \text{ and}\\
(x,t) \in \T \times (0,1] &\mapsto O_t(x) \in \mathrm{SO}(2) 
\end{align}
such that 
\begin{alignat}{2}
	A_0(x) &= A(x)  &\quad &\text{for all $x\in \T$ and}\\
	A_t(x) &= C_t(R_\alpha(x)) O_t(x) C_t(x)^{-1} &\quad &\text{for all $(x,t) \in \T \times (0,1]$.}
\end{alignat}
\end{theorem}

Thus, for each $t \in (0,1]$, the cocycle $(R_\alpha,A_t)$ is product bounded, with a corresponding bound that depends continuously on $t$ and tends to infinity as $t \to 0$.
By reparametrization, this convergence can be made arbitrarily slow.
More precisely, we have the following statement:

\begin{lemma}\label{l.slow_ABD}
Let $\alpha \in \R \setminus \Q$ and $c>0$.
Let $M \colon (0,1] \to [0,\infty)$ be any continuous strictly decreasing function such that $M(t) \to \infty$ as $t \to 0$.
Then there exists a continuous map
\begin{equation}
(x,t) \in \T \times [0,1] \mapsto B_t(x) \in \mathrm{SL}(2,\R)
\end{equation}
such that 
\begin{equation}\label{e.prescribed_LE}
\lambda_1(R_\alpha,B_0,\mathrm{Leb}) = c
\end{equation}
and, for all $t \in (0,1]$, $x \in \T$, and $n \ge 0$,
\begin{equation}\label{e.B_and_M}
\| B_t^{(n)}(x) \| \le e^{M(t)} 
\end{equation}
where $B_t^{(n)} = (B_t)^{(n)}_{R_\alpha}$. 
\end{lemma}

\begin{proof}
Given $\alpha$ and $c$, let $A$ be given by \cref{p.NUH}.
We apply \cref{t.ABD} obtaining the functions $A_t(x)$, $C_t(x)$, and $O_t(x)$.
For each $(x,t) \in \T \times [0,1]$, let
\begin{equation}
\hat{A}_t(x) \coloneqq C_0(R_\alpha(x))^{-1} A_t(x) C_0(x) \, ,
\end{equation}
which is a continuous function to $\mathrm{SL}(2,\R)$.
The cocycle $(R_\alpha,\hat{A}_0)$ is conjugate to $(R_\alpha,A_0)$ and so it has Lyapunov exponent $c$.
On the other hand, $\hat{A}_1 = O_1$ takes values in $\mathrm{SO}(2)$.
Furthermore, for each $t \in (0,1]$, the cocycle $(R_\alpha,\hat{A}_t)$ is product-bounded; in fact,
\begin{equation}
\sup_{x \in X}  \sup_{n \ge 0} \| \hat{A}_t^{(n)}(x) \| \le 
\sup_{x \in X} \|C_0(x)^{-1} C_t(x)\|^2 \, .
\end{equation}
Let
\begin{equation}
\hat{M}(t) \coloneqq 2 \log  \sup_{x \in X}  \sup_{s \in [t,1]} \|C_0(x)^{-1} C_t(x)\| \, .
\end{equation}
Then $\hat{M} \colon (0,1] \to [0,\infty)$ is continuous, decreasing, and $\hat{M}(1) = 0$.
Given the function $M$ as in the statement of the \lcnamecref{l.slow_ABD}, let $\phi \colon [0,1] \to [0,1]$ be an increasing homeomorphism such that $\hat{M}(\phi(t)) \le M(t)$ for all $t \in (0,1]$.
Then the function $B_t(x) \coloneqq \hat{A}_{\phi(t)}(x)$ has the required properties. 
\end{proof}

\section{Sturmian measures}\label{s.sturmian}

Sturmian sequences and measures have a long history: see \cite[Chapter~2]{Lothaire}, \cite[Chapter~6]{Fogg}, and references therein. 
Sturmian sequences are usually viewed as elements of the symbolic space $\{0,1\}^\N$, and the corresponding measures are shift-invariant.
However, since we are working with the doubling map on the circle, we will introduce Sturmian measures in the corresponding setting. 
Our exposition is self-contained, but it definitely overlaps \cite{Veerman,BSentenac} (e.g., compare \cref{f.Sentenac} below and \cite[Fig.~13]{Veerman}).
 
Let $\alpha$ be an irrational number, which will be fixed from now on.

\subsection{A devil staircase}

We use the standard notations for the floor and ceiling functions, that is, $\lfloor x \rfloor \coloneqq \max \{ n \in \Z \st n \le x\}$ and $\lceil x \rceil \coloneqq \min \{ n \in \Z \st n \ge x\}$.
Let $F \colon \R \to \R$ be defined as follows:
\begin{equation}\label{e.F_def}
F(x) \coloneqq \sum_{n=0}^\infty 2^{-n-1} \lfloor x+n\alpha \rfloor \, .
\end{equation}
The following properties are immediate: 
\begin{align}
F(x+1)&= F(x)+1 \, , \\ 
2F(x) &= \lfloor x \rfloor + F(x+\alpha) \, . \label{e.doubling_F}
\end{align}
The function $F$ is right-continuous and (since $\alpha$ is irrational) strictly increasing. 
Define another function:
\begin{equation}
f(x) \coloneqq \sum_{n=0}^\infty 2^{-n-1} (\lceil  x+n\alpha \rceil -1) \, .
\end{equation}
Then $f \le F$ and $f$ has the same properties listed above, except that 
\begin{equation}\label{e.doubling_f}
2f(x) = \lceil x \rceil - 1 + f(x+\alpha) 
\end{equation}
and $f$ is left-continuous. 
Furthermore, $F(x) = f(x)$, unless $x = -n\alpha + m$ for some integers $n \ge 0$ and $m$.
The values of the two functions at these points satisfy the relation:
\begin{equation}\label{e.F_gap}
F(-n\alpha+m) = f(-n\alpha+m) + 2^{-n-1} \, , \quad n , m \in \Z, \  n \ge 0 \, .
\end{equation}

Let $\tilde{h} \colon \R \to \R$ be defined by the following condition:
\begin{equation}
\tilde{h}(y) = x \quad \Leftrightarrow \quad 
f(x) \le y \le F(x) \, ;
\end{equation}
see \cref{f.devil}.
Then $\tilde{h}$ is increasing, continuous, and satisfies $\tilde{h}(x+1) = \tilde{h}(x)+1$.
As a consequence of equalities \eqref{e.doubling_F} and \eqref{e.doubling_f}, we have the following property:
\begin{equation}\label{e.tilde_h_property}
\tilde{h}(y) = x \not\in \Z \quad \Rightarrow \quad
\tilde{h}(2y) = \lfloor x \rfloor + x + \alpha \, .
\end{equation}

\begin{figure} 
\includegraphics[width=.55\textwidth]{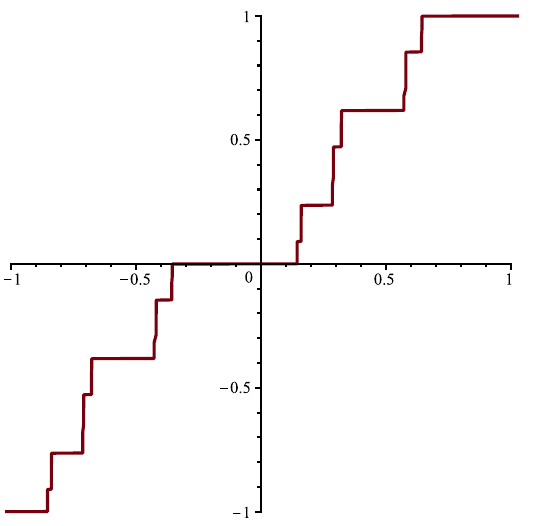}
\caption{Graph of the function $\tilde{h}$ if $\alpha = \frac{3-\sqrt{5}}{2}$.}\label{f.devil}
\end{figure}

\begin{remark}\label{r.mechanical}
Using summation by parts, we can rewrite $F$ as
\begin{equation}
F(x) = \lfloor x \rfloor + \sum_{n=0}^\infty 2^{-n-1} \big( \lfloor x+(n+1)\alpha \rfloor - \lfloor x+n\alpha \rfloor \big) \, ;
\end{equation}
similarly for $f$.
Therefore, if $0<\alpha<1$, then the binary expansions of the fractional parts of $F(x)$ and $f(x)$ constitute ``rotation sequences'' \cite[p.~151]{Fogg}, ``mechanical words'' \cite[p.~53]{Lothaire}, or ``optimal sequences'' \cite[p.~550]{Veerman}. 
\end{remark}

\subsection{Down to the circle}\label{ss.down}

Recall that $\T \coloneq \R/\Z$. 
Let $\pi \colon \R \to \T$ denote the quotient projection.
The Haar measure on the circle $\T$ is  $\mathrm{Leb}_{\T} = \pi_*(\mathrm{Leb}_{[0,1]})$, where $\mathrm{Leb}_{[0,1]}$ denotes Lebesgue measure on the unit interval and the star denotes push-forward.

We consider two dynamical systems on $\T$, the doubling map $D$ 
and the irrational rotation $R = R_\alpha$, which can be defined by the formulas
\begin{equation}
D(\pi(x)) \coloneqq \pi(2x) \, , \quad
R(\pi(x)) \coloneqq \pi(x+\alpha) \, , \quad x\in \R \, .
\end{equation}

Let $h \colon \T \to \T$ be defined by 
\begin{equation}\label{e.def_h}
h(\pi(x)) \coloneqq \pi(\tilde{h}(x)) \, , \quad x \in \R \, ,
\end{equation}
where $\tilde{h}$ is as in the previous subsection.
Then $h$ is continuous and has degree $1$; in particular, it is surjective.
Furthermore, 
\begin{equation}\label{e.Ffh}
h \circ \pi \circ F = h \circ \pi \circ f = \pi \, .
\end{equation}

For each integer $n \ge 0$, define a set  $I_n \subseteq \T$ by
\begin{equation}\label{e.def_I}
	I_n \coloneqq \pi \big( (f(-n\alpha),F(-n\alpha)) \big) \, .
\end{equation}
Equivalently, $I_n$ is the interior of the set $h^{-1}(R^{-n}(0))$.
By property \eqref{e.F_gap}, $I_n$ is an open interval of length $2^{-n-1}$
(where by an \emph{interval} we mean a connected proper subset of $\T$, and \emph{length} is Haar measure).
These intervals have pairwise disjoint closures. 
Property~\eqref{e.tilde_h_property} implies:
\begin{equation}\label{e.conjugacy}
	h \circ D = R \circ h \quad \text{on the semicircle } \T \setminus I_0 \, .
\end{equation}
In particular,
\begin{equation}\label{e.tower_property}
	D(I_n) = I_{n-1} \quad \text{for all } n > 0 \, .
\end{equation}
See \cref{f.Sentenac}.

\tikzset{
  pics/carc/.style args={#1:#2:#3}{
    code={
      \draw[pic actions] (#1:#3) arc(#1:#2:#3);
    }
  }
}
\begin{figure} 
	\begin{tikzpicture}[scale=.8]
		\draw (0,0) pic[Parenthesis-Parenthesis]{carc=232.2:412.2:2.5cm}; 
		\draw ( .791*2.75,-.612*2.75) node{$I_0$};
		\draw (0,0) pic[Parenthesis-Parenthesis]{carc=116.1:206.1:2.5cm};
		\draw (-.946*2.75, .324*2.75) node{$I_1$};
		\draw (0,0) pic[Parenthesis-Parenthesis]{carc= 58.1:103.1:2.5cm};
		\draw ( .164*2.75, .986*2.75) node{$I_2$};
		\draw (0,0) pic[Parenthesis-Parenthesis]{carc=209.0:231.5:2.5cm};
		\draw (-.763*2.75,-.647*2.75) node{$I_3$};
		\draw (0,0) pic[Parenthesis-Parenthesis]{carc=104.5:115.8:2.5cm};
		\draw (-.344*2.75, .939*2.75) node{$I_4$};
		
		\draw[->] (3.75,0)--(4.5,0) node[midway,above]{$h$};
		
		\begin{scope}[xshift=7.75cm]
			\draw (0,0) circle(1);
			\fill (1    ,0    ) circle(.05) node[right]{$0$};
			\fill (-.737,-.675) circle(.05) node[left] {$R_\alpha^{-1}(0)$};
			\fill ( .087, .996) circle(.05) node[above]{$R_\alpha^{-2}(0)$};
			\fill ( .608,-.794) circle(.05) node[right]{$R_\alpha^{-3}(0)$};
			\fill (-.985, .174) circle(.05) node[left] {$R_\alpha^{-4}(0)$};
		\end{scope}
		
	\end{tikzpicture}
	\caption{Intervals $I_n$ and their images $h(I_n) = \{R_\alpha^{-n}(0)\}$ for $\alpha = \frac{3-\sqrt{5}}{2}$.}
	\label{f.Sentenac}
\end{figure}
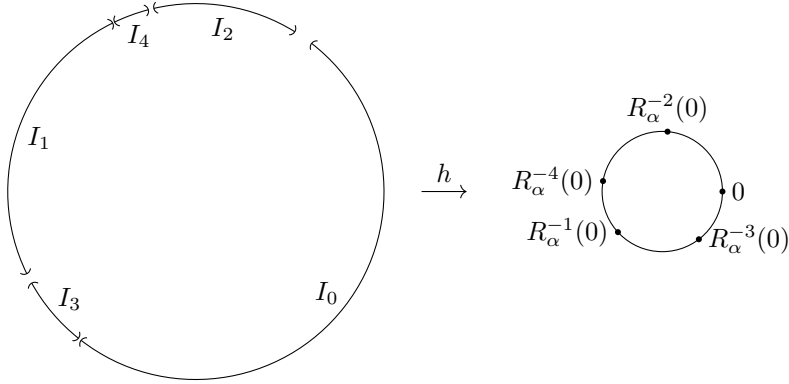

Consider the set
\begin{equation}\label{e.def_K}
	K \coloneqq \T \setminus \bigsqcup_{n \ge 0} I_n \, .
\end{equation}
Equivalently, $K = \bigcap_{n \ge 0} D^{-n}(\T \setminus I_0)$.
Then $K$ is a Cantor set and satisfies $D(K) = K$ and $h(K) = \T$. 

We define the \emph{Sturmian measure with parameter $\alpha$} as 
\begin{equation}
\nu \coloneqq (\pi\circ F)_* (\mathrm{Leb}_{[0,1]}) =  (\pi \circ f)_* (\mathrm{Leb}_{[0,1]}) \, .
\end{equation}
This is a Borel probability measure supported on $K$.
By \eqref{e.Ffh}, this measure satisfies $h_* \nu = \mathrm{Leb}_{\T}$.
It follows from \eqref{e.conjugacy} that $\nu$ is $D$-invariant.
Additionally, as a consequence of the fact that $R$ is uniquely ergodic, $\nu$ is the unique $D$-invariant probability measure whose support is contained in the semicircle $\T \setminus I_0$.

\begin{remark}
If the parameter $\alpha$ is rational, then the construction above can be performed with some adaptations, yielding a measure $\nu$ whose support is a periodic orbit contained in a semicircle. See \cite{BSentenac}.
\end{remark}

\section{Proof of the theorem}

Fix numbers $\alpha \in \R \setminus \Q$ and $c>\epsilon>0$.

\subsection{Step 1: Defining the cocycle}

The \emph{arc-length metric} on the circle $\T$ is: 
\begin{equation}\label{e.d_def}
\mathrm{d}(x,y) \coloneqq \min \big\{|r - s| \st r, s \in \R, \ r \equiv x \bmod \Z, \ s  \equiv y \bmod \Z \big\} \, .
\end{equation}
Using the sets $I_n$ and $K$ constructed in the previous section, we define an auxiliary function $\phi \colon \T \to \R$ as follows:
\begin{equation}\label{e.phi_def}
\phi(x) \coloneqq 
\begin{cases}
2^{n+2} \mathrm{d}(x,\partial I_n)	&\text{if } x \in I_n \, , \ n \ge 0, \\
0							&\text{if } x \in K  \, .
\end{cases}
\end{equation}
See \cref{f.phi}.
By \eqref{e.tower_property}, for each $n>0$, the map $D$ sends the interval $I_n$ onto the interval $I_{n-1}$ doubling arc-length. 
Therefore, for each $n \ge 0$ and $x \in I_n$,
\begin{equation}\label{e.constancy_phi}
\phi(x) = \phi(D x) = \cdots = \phi(D^n x) =  4 \mathrm{d}(D^n x, \partial I_0) \, .
\end{equation}

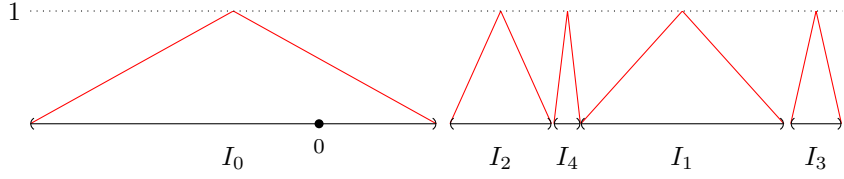
\begin{figure} 
\begin{tikzpicture}[scale=3] 

\begin{scope}
\clip(-1.28,0) rectangle (2.32,0.5);
\draw[MapleRed,thick](-1.28,0) -- (-.38,0.5) -- (.52,0);
\draw[MapleRed,thick](1.16,0) -- (1.61,0.5) -- (2.06,0);
\draw[MapleRed,thick](.58,0) -- (.805,0.5) -- (1.03,0);
\draw[MapleRed,thick](2.09,0) -- (2.2025,0.5) -- (2.315,0);
\draw[MapleRed,thick](1.0405,0) -- (1.101,0.5) -- (1.158,0);
\end{scope}

\draw[Parenthesis-Parenthesis](-1.28,0) -- (.52,0);
\draw[Parenthesis-Parenthesis](1.16,0) -- (2.06,0);
\draw[Parenthesis-Parenthesis](.58,0) -- (1.03,0);
\draw[Parenthesis-Parenthesis](2.09,0) -- (2.315,0);
\draw[Parenthesis-Parenthesis](1.0405,0) -- (1.158,0);

\node(textI0) at (-.38,-.15) {$I_0$};
\node(textI1) at (1.61,-.15) {$I_1$};
\node(textI2) at (.805,-.15) {$I_2$};
\node(textI3) at (2.2025,-.15) {$I_3$};
\node(textI4) at (1.101,-.15) {$I_4$};

\fill (0,0) circle(.02);
\node(origin) at (0,-.1) {\footnotesize $0$};

\draw[dotted] (-1.28,0.5) node[left] {$1$} --(2.32,0.5);

\end{tikzpicture}
\caption{Graph of the discontinuous function $\phi$ on the first five intervals, for $\alpha = \frac{3-\sqrt{5}}{2}$.}\label{f.phi}
\end{figure}

\begin{remark}
At this point we could define $\tilde A(x) \coloneqq B_{\phi(x)}(h(x))$, 
where $B$ comes from \cref{l.slow_ABD} (the choice of $M$ being immaterial).
Then we would obtain a $\mathrm{SL}(2,\R)$-cocycle such that
$\lambda_1(D,\tilde{A},\nu) = c$ and $\lambda_1(D,\tilde{A},\mu) = 0$ for all ergodic measures $\mu\neq\nu$.
However, this cocycle would be discontinuous.
The continuous cocycle we are looking for consists of a suitably modified version of $\tilde{A}$.
\end{remark}

For each integer $n \ge 0$, let
\begin{equation}\label{e.ell_def}
\ell(n) \coloneqq \lfloor (n+2)^{1/4} \rfloor \, .
\end{equation}
Define another auxiliary function $\psi$ on $\T$ as 
\begin{equation}\label{e.psi_def}
\psi(x) \coloneqq 
\begin{cases}
\phi(x)/\ell(n)	&\text{if } x \in I_n \, , \ n \ge 0, \\
0				&\text{if } x \in K  \, .
\end{cases}
\end{equation}

\begin{lemma}
The function $\psi$ is continuous. 
\end{lemma}

\begin{proof}
For each $n \ge 1$, let $\psi_n \colon \T \to \R$ be the piecewise affine function defined as $\psi_n(x) \coloneqq \psi(x)$ if $ N(x) \le n$ and $\psi_n(x) \coloneqq 0$ otherwise. 
Then $|\psi(x) - \psi_n(x)| \le \frac{1}{\ell(n)}$ for every $x$.
Therefore, $\psi$ is a uniform limit of a sequence of continuous functions, and so it is continuous. 
\end{proof}

For each positive integer $n$, let
\begin{equation}\label{e.delta_def}
\delta(n) \coloneqq \min \big\{ 4 \mathrm{d}(x, \partial I_0) \st x \in I_1 \sqcup \cdots \sqcup I_n \big\} \, .
\end{equation}
Then $1 > \delta(1) \ge \delta(2) \ge \cdots$ and $\delta(n) \to 0$ as $n \to \infty$.
Additionally, since $I_0 \cap D^{-1}(I_n) = I_{n+1} + \tfrac{1}{2}$, we have
\begin{equation}
\label{e.delta_property_3}
x \in I_0 \cap D^{-1}(I_n) \quad \Rightarrow \quad 
4 \mathrm{d}(x, \partial I_0) = 4 \mathrm{d}(x + \tfrac{1}{2}, \partial I_0)  \ge \delta(n+1) \, .
\end{equation}

The values $\frac{\delta(n+1)}{\ell(n)}$ form a strictly decreasing sequence in $[0,1]$ tending to $0$. 
Therefore, we can choose a continuous strictly decreasing function $M \colon (0,1] \to [0,\infty)$ such that
\begin{equation}\label{e.M_control_pts}
M \left( \frac{\delta(n+1)}{\ell(n)}\right) = \frac{\epsilon \,  (n+2)^{1/4}}{\sqrt{2}}  \, . 
\end{equation}
Let $B_t(x)$ be the matrix function (depending on this function $M$) given by \cref{l.slow_ABD}.

Define $A \colon \T \to \mathrm{SL}(2,\R)$ as:
\begin{equation}\label{e.def_A}
A(x) \coloneqq B_{\psi(x)}(h(x)) \, ,
\end{equation}
where $h$ comes from \eqref{e.def_h}.
Since $h$, $\psi$, and $B$ are continuous functions, $A$ is continuous as well.

\subsection{Step 2: Bounding Lyapunov exponents}

The cocycle we have just constructed has the prescribed Lyapunov exponent with respect to the Sturmian measure $\nu$:

\begin{lemma}\label{l.c}
$\lambda_1(D,A,\nu) = c$.
\end{lemma}

\begin{proof}
Since $\psi$ vanishes on $K$, it follows from definition \eqref{e.def_A} and  property \eqref{e.conjugacy} that 
\begin{equation}
A^{(n)}_D(x) = (B_0)^{(n)}_R(h(x)) \quad \text{for all $x \in K$ and $n \ge 0$.}
\end{equation}
Since $\nu(K)=1$ and $h_* \nu = \mathrm{Leb}$, we obtain $\lambda_1(D,A,\nu) = \lambda_1(R,B_0,\mathrm{Leb})$, which by equation~\eqref{e.prescribed_LE} equals $c$.
\end{proof}

\begin{lemma}\label{l.key}
Let $n\ge 0$, $m\ge 0$.
Then, for all $x \in I_n \cap D^{-(n+1)}(I_m)$, we have
\begin{equation}\label{e.key}
\log \|A^{(n+1)}(x) \|  \le 
\epsilon \, \sqrt{\frac{n+m+2}{2}} \, .
\end{equation}
\end{lemma}

\begin{proof}
Fix $x \in I_n \cap D^{-(n+1)}(I_m)$.
Write
\begin{equation}\label{e.climbing}
A^{(n+1)}(x) = A(x_0) A(x_1) \cdots A(x_n) \, , \quad
\text{where } x_k \coloneqq D^{n-k} x \, .
\end{equation}
Note that $x_k \in I_k$ for each $k \in \{0,\dots,n\}$.

Let $t \coloneqq 4 \mathrm{d}(x_0 , \partial I_0)$. 
By property~\eqref{e.constancy_phi}, $\phi(x_k) = t$ for all $k \in \{0,\dots,n\}$.
Since $x_0 \in I_0 \cap D^{-1}(I_m)$, by property~\eqref{e.delta_property_3} we have 
\begin{equation}\label{e.win}
t \ge \delta(m+1) \, .
\end{equation}

Let $s \coloneqq \ell(n)$.
Note that $s^4-2 \le n$.
Define numbers $n_1 < n_2 < \cdots < n_{s+1}$ as:
\begin{equation}
n_j \coloneqq 
\begin{cases}
0		&\text{if $j=1$,}\\
j^4-2	&\text{if $2 \le j \le s$,}\\
n+1		&\text{if $j=s+1$.}
\end{cases}
\end{equation}
Then, for each $j \in \{1,\dots,s\}$,
\begin{equation}\label{e.constancy_ell}
n_j \le k < n_{j+1} \quad \Rightarrow \quad \ell(k) = j \, .
\end{equation}
Rewrite the product \eqref{e.climbing} as
\begin{gather}
\label{e.P_factorization}
A^{(n+1)}(x) = P_1 P_2 \cdots P_s \, , \quad \text{where} \\
P_j \coloneqq A(x_{n_j}) A(x_{n_j+1}) \cdots A(x_{n_{j+1}-1}) \, . 
\end{gather}

Fix $j \in \{1,\dots,s\}$, and let us bound $\|P_j\|$.
It follows from property \eqref{e.constancy_ell} that 
\begin{equation}\label{e.constancy_psi}
n_j \le k < n_{j+1} \quad \Rightarrow \quad \psi(x_k) = \frac{t}{j} \, .
\end{equation}
Using the fact that $h|_{I_k} \equiv R^{-k}(0)$, definition \eqref{e.def_A} gives:
\begin{equation}
n_j \le k < n_{j+1} \quad \Rightarrow \quad A(x_k) = B_{t/j}(R^{-k} (0)) \, .
\end{equation}
Therefore
\begin{equation}
P_j = B_{t/j}(R^{-n_j} (0)) B_{t/j}(R^{-n_j+1} (0)) \cdots B_{t/j}(R^{-n_{j+1}+1} (0)) \, .
\end{equation}
By the fundamental property \eqref{e.B_and_M},
\begin{equation}
\|P_j\| \le e^{M(t/j)} \, .
\end{equation}

Now we can bound $\log \|A^{n+1}(x) \| \le \sum_{j=1}^s M(t/j)$.
Since $M(\mathord{\cdot})$ is decreasing,
\begin{equation}
\log \|A^{n+1}(x) \| 
\le s M(t/s) 
\overset{\eqref{e.win}}{\le}
(n+2)^{\frac{1}{4}}  M\left(\frac{\delta(m+1)}{\ell(n)}\right) \, ,
\end{equation}
Since $\delta(\mathord{\cdot})$ is decreasing and $\ell(\mathord{\cdot})$ is increasing,
\begin{equation}
\log \|A^{n+1}(x) \| 
\le (n+m+2)^{\frac{1}{4}}  M\left(\frac{\delta(n+m+1)}{\ell(n+m)}\right) 
\overset{\eqref{e.M_control_pts}}{=} \epsilon \, \sqrt{\frac{n+m+2}{2}} \, . \qedhere
\end{equation}
\end{proof}

\begin{lemma}\label{l.LE_bound}
Let $\mu$ be a $D$-invariant ergodic Borel probability measure such that $\mu \neq \nu$.
Then 
\begin{equation}\label{e.LE_bound}
\lambda_1(D,A,\mu) \le \epsilon \sqrt{\mu(I_0)} \, .
\end{equation}
\end{lemma}

\begin{proof}
Since $\nu$	is the unique $D$-invariant probability measure supported on $\T \setminus I_0$ 
and $\mu \neq \nu$, we have $\mu(I_0)>0$. 
Since $\mu$ is ergodic, the orbit $(D^k x)_{k\ge 0}$ of $\mu$-almost every point $x \in \T$ hits the set $I_0$ infinitely many times;
furthermore, if $k_0 < k_1 < k_2 < \cdots$ denote the hitting times, then
\begin{equation}\label{e.frequency}
\lim_{i \to \infty} \frac{i}{k_i} 
= \mu(I_0) \, ,
\end{equation}
by Birkhoff's theorem. 
Fix a point $x$ for which these properties hold. 
For each $i\ge 0$, let
\begin{equation}
x_i \coloneqq D^{k_i + 1} x
\quad \text{and} \quad 
n_i \coloneqq k_{i+1} - k_i \, .
\end{equation}
Note that $x_i \in I_{n_i-1}$.
Then we can apply \cref{l.key} with $(x_i, n_i-1, n_{i+1}-1)$ in place of $(x,n,m)$, obtaining 
\begin{equation}
\log\|A^{(n_i)}(x_i)\| \le \epsilon \, \sqrt{a_i}  \, , 
\quad \text{where } 
a_i \coloneqq \frac{n_i+n_{i+1}}{2} \, .
\end{equation}
Letting $j \ge 1$, consider the factorization
\begin{equation}
A^{(k_j + 1)}(x) = A^{(n_{j-1})}(x_{j-1})  \cdots A^{(n_1)}(x_1) \, A^{(n_0)}(x_0) \,  A^{(k_0+1)}(x) \, ,
\end{equation}
Let $C \coloneqq \log \|A^{(k_0+1)}(x)\|$.
Then:
\begin{align}
\log \|A^{(k_j + 1)}(x)\| 
&\le C + \epsilon \left( \sqrt{a_0} + \cdots + \sqrt{a_{j-1}} \right) \\
&\le C + \epsilon \sqrt{j} \, \sqrt{a_0 + \cdots + a_{j-1}}   \, ,
\end{align}
by the Cauchy--Schwarz inequality. 
Since $a_0 + \cdots + a_{j-1} < n_0 + \cdots + n_j \le k_{j+1}$, we obtain
\begin{equation} 
\frac{\log \|A^{(k_j + 1)}(x)\|}{k_j}  
\le 
\frac{C}{k_j} + 
\epsilon \, \frac{j}{k_j} \, \sqrt{\frac{k_{j+1}}{j}} \, .
\end{equation}
For $\mu$-a.e.\ $x$, as $j \to \infty$, the left hand side converges to the Lyapunov exponent $\lambda_1(D,A,\mu)$  while, thanks to \eqref{e.frequency}, the right hand side converges to $\epsilon \sqrt{\mu(I_0)}$.
So we obtain the announced inequality \eqref{e.LE_bound}. 
\end{proof}

\begin{proof}[Proof of \cref{t.main}]
Statement~\eqref{i.main1} was already checked in \cref{l.c}, 
while statement~\eqref{i.main2} follows immediately from \cref{l.LE_bound}.
As for the last statement, suppose that $\mu_k \to \nu$ (weakly), with each $\mu_k \neq \nu$.
Since $\nu(\partial I_0) = 0$, it follows that $\mu_k(I_0) \to \nu(I_0) = 0$ (see e.g.\ \cite[Thrm.~II.6.1(e)]{Partha}). 
Now \cref{l.LE_bound} gives $\lambda_1(D,A,\mu_k) \to 0$.
\end{proof}

\section{Questions}

Does there exist a continuous $\mathrm{SL}^{\smallpm}(2,\R)$-cocycle over some expanding or hyperbolic base dynamics admitting a unique ergodic measure with non-zero Lyapunov exponents? 
Can such a cocycle have trivial periodic data in the sense of \cite{Kalinin}? 
Perhaps it is possible to construct an example by adapting a few ideas from \cite{Bochi_mono}, 
but the author is currently not able to overcome all difficulties. 
For a related open problem, see \cite[{\S}4.5]{BPS}.

In this paper, Sturmian measures were convenient for two reasons: first, the relation with irrrational rotations allowed us to use \cref{p.NUH}; second, the complement of the support of a Sturmian measure has an especially simple structure. It is conceivable that the construction can be adapted to other invariant measures. For what ergodic measures is the statement of \cref{t.main} valid? Can we characterize such measures allowing arbitrary expanding or hyperbolic base dynamics?

A more challenging problem, in the spirit of the flexibility program (see \cite{BKRH}), is as follows:
Fixed an expanding or hyperbolic base dynamics $T$, is it possible to fully describe the collection of the functions $\lambda_1(T,A,\mathord{\cdot}) \colon \mathcal{M}_T \to [0,\infty)$, where $A$ varies in $C^0(X,\mathrm{SL}^{\smallpm}(2,\R))$? Those functions are upper-semicontinuous and affine; are there any other restrictions?

Is it possible for derivative cocycles to exhibit the behavior described in this paper?
Does there exist a $C^1$-diffeomorphism $f$ of a compact surface and a nontrivial homoclinic class $\Lambda$ such that $f|_{\Lambda}$ admits an ergodic measure $\nu$ whose Lyapunov exponents are isolated from the Lyapunov exponents of all other ergodic measures $\mu$ in the sense that 
\begin{equation}
\lambda_1(f,\nu) > \epsilon \ge \lambda_1(f,\mu) \ge \lambda_2(f,\mu) \ge -\epsilon > \lambda_2(f,\nu) \quad \text{?}
\end{equation}

In a different direction, is it possible to identify the optimal modulus of continuity of the cocycle under which the Periodic Approximation Theorem~\ref{t.PAT} holds?


\medskip

\noindent \textbf{Acknowledgement.} 
I thank the referee for his/her suggestions and corrections.


\end{document}